\documentclass{conm-p-l}

\usepackage{amssymb}

\usepackage{}
\usepackage{latexsym}
\usepackage{amsmath,amssymb,amsthm,amsfonts,latexsym}
\usepackage[bookmarks]{hyperref}

\newtheorem{theorem}{Theorem}[section]

\newtheorem{corollary}[theorem]{Corollary}
\newtheorem{proposition}[theorem]{Proposition}

\theoremstyle{definition}
\newtheorem{definition}[theorem]{Definition}

\theoremstyle{remark}
\newtheorem{remark}[theorem]{Remark}

\numberwithin{equation}{section}

\let\=\partial

\theoremstyle{definition}
 
\newcommand{\R} {\mathbb{R}}

\newcommand{\hn}{\mathcal{H}}

\newcommand{\rb}{(r_{\int})}

\newcommand{\A}{\mathcal{A}}

\newcommand{\ep}{\varepsilon}
\newcommand{\pa}{\partial}

\begin{document}

 \title[short text for running head]{full title}
\title{On the local geometry of definably stratified sets}


\author{David Trotman}
\address{Aix Marseille Univ., CNRS, Centrale Marseille, Institut Math\'ematique de Marseille, 13453 Marseille, France.}
\curraddr{}
\email{david.trotman@univ-amu.fr}
\thanks{Visits  to the Jagiellonian University in Krakow and the University of Provence in Marseille, during which this work developed, were supported by a PHC Polonium project.
The research began while both authors were resident at the Fields Institute in Toronto during the thematic program, ``O-minimal Structures and Real Analytic  Geometry". 
We gratefully acknowledge the support of the French Ministry of Foreign and European Affairs (MAEE), the French Ministry of Higher Education and Research (MESR), the Fields Institute and the French Embassy in Ottawa.}

\author{Guillaume Valette}
\address{Instytut Matematyczny PAN, ul. \'Sw. Tomasza 30, 31-027 Krak\'ow, Poland}
\curraddr{}
\email{gvalette@impan.pl}
\thanks{}

\subjclass[2000]{Primary }
\subjclass[2010]{Primary 03C64, 32S15, 53B25, 58A35; Secondary 32S60 }

\date{}

\begin{abstract}
We prove that a theorem of Paw\l ucki, showing that  Whitney regularity for a subanalytic
 set with a smooth singular locus of codimension one implies  the set is a finite union of $C^1$ manifolds with boundary,
 applies to definable sets in polynomially bounded o-minimal structures.
We give a refined version of Paw\l ucki's theorem for arbitrary
o-minimal structures, replacing Whitney (b)-regularity by a quantified version, and we prove related results concerning normal cones and continuity of the density.
We  analyse two counterexamples to the  extension of Paw\l ucki's
theorem to definable subsets in general o-minimal structures, and to several other statements valid for subanalytic sets. In particular we give the first example of a Whitney $(b)$-regular definably stratified set such that the density is not continuous along a stratum.
\end{abstract}

\maketitle


\section{Introduction}

We recall the statement of a striking and useful  theorem of W. Paw\l ucki.

 \medskip
 \begin{theorem} (Paw\l ucki  \cite{pa})
  {\it A subanalytic stratified set $S$ with a smooth singular locus $L$ of codimension $1$ is Whitney $(b)$-regular if and only if $S$ is a finite union of $C^1$ manifolds-with-boundary $S_1, \dots, S_k$ each of which has boundary $L$.}
 \end{theorem}
 
 \medskip
 Paw\l ucki used this theorem to prove a version of Stokes' theorem for subanalytic sets, later generalised by \L ojasiewicz \cite{loj} and by Valette \cite{va4}.  It has also been used in various ways in papers of  Bernig \cite{be}, Chossat-Koenig \cite{ck}, Comte \cite{c}, Jeddi \cite{j}, Orro-Pelletier \cite{op}  and Valette \cite{va3}, among others.


\medskip

It is natural to consider possible generalisations of Paw\l ucki's theorem to definable sets in o-minimal structures. We prove in section 3 that Paw\l ucki's theorem applies to definable sets in {\it polynomially bounded} o-minimal structures (Corollary 3.11).
Furthermore we give in Corollary 3.10 a refined version of Paw\l ucki's theorem which applies to arbitrary o-minimal structures, replacing Whitney $(b)$-regularity by a quantified version first described by the second author in his thesis \cite{va1} (directed by the first author). We also generalise a theorem of Hironaka \cite{h} stating that Whitney regularity of an analytic variety controls the normal cone structure, and a theorem of Comte \cite{c} stating that Kuo's ratio test ensures continuity of the density for subanalytic sets.

In section 4 we analyse two counterexamples to the natural generalisation of Paw\l ucki's theorem for definable subsets in  o-minimal structures which are not polynomially bounded. These also provide counterexamples to several classical results describing relations between well-known equisingularity conditions for subanalytic strata, including the above theorem of Hironaka and a result of Navarro Aznar and the first author \cite{nt}. We also  obtain the first example of a Whitney $(b)$-regular definable stratified set for which the density is not continuous along each stratum; note that the density is continuous along strata of a $(b)$-regular definable stratified set when the o-minimal structure is polynomially bounded \cite{va2}, \cite{nv}.

\medskip


\bigskip
\section{Definitions.}


 Let $k \geq 2$ be an integer. Let $S$ be a closed stratified subset of ${\bf R}^n$, such that the strata  are differentiable submanifolds of class $C^k$. For each stratum $X$ of $S$ denote by
$C_{X}S$ the {\bf normal cone of $S$ along $X$}, i.e. the restriction to $X$ of the
closure of the set $\{(x, \mu(x\pi(x))) : x \in S \setminus X \} \subset {\bf R}^n \times S^{n-1}$, where $\pi$ is
the local canonical retraction onto $X$ (recall that this is defined on a neighbourhood of a differentiable manifold of class $C^2$ by taking the unique nearest point on $X$), $\mu(x)$ is the unit vector $\frac{x}{ \vert\vert x \vert \vert }$,
 and $uv$ denotes the vector $v - u$, where $u$ and $v$ are points of ${\bf R}^n$. In fact $C_XS$
is a union of normal cones $C_X Y_i$, where the $\{Y_i \}$ are the strata of $S$
whose closures contain $X$. The following two properties of the normal cone express that $S$ behaves well along the stratum $X$ (at the point $x$).

\medskip
\begin{definition}
{\bf Condition $(n)$:} The fibre $(C_XS)_x$ of $C_XS$ at a point $x$ of $X$ equals the
tangent cone $C_x(S_x)$ to the fibre $S_x = S \cap \pi^{-1}(x)$ of $S$ at $x$.
\end{definition}

\begin{definition}
{\bf Normal pseudo-flatness $(npf)$  \cite{h}:}  The stratified set $S$ is said to be {\bf normally pseudo-flat along $X$} when the projection $p : C_XS \rightarrow X$ is open.
\end{definition}

When a stratification satisfies simultaneously two regularity conditions, say Whitney
$(a)$-regularity and $(n)$, we write that it is $(a+n)$-regular. Subanalytic
stratifications satisfying $(a+n)$ or $(npf)$ have normal cones with good
behaviour from the point of view of the dimension of the fibres. In fact
they satisfy the condition
$$dim(C_XS)_x \leq dim \, S - dim\, X - 1.  \quad\quad (*)$$

\noindent This is obvious for $(a + n)$, while for $(npf)$ it follows from \cite{ot2}. For differentiable stratifications one may not always be able to define the dimension.
Despite the bound $(*)$, the tangent cone $C_x(S_x)$ to the fibre
$S_x = S \cap \pi^{-1}(x)$ (hence the fibre $(C_XS)_x$ of the normal cone, assuming $(n)$)
can be quite arbitrary :  work of Ferrarotti,  Fortuna and Wilson \cite{ffw}
shows that each closed semi-algebraic cone of codimension $ \geq 1$ is realised
as the tangent cone at a point of a certain real algebraic variety,
while Kwieci\'nski and the first author showed that {\it every} closed cone is realised
as the tangent  cone at an isolated singularity of a certain $C^{\infty}$ $(b)$-regular
stratified set \cite{kt}.

 Hironaka proved in \cite{h} that  Whitney $(b)$-regular stratifications  of any analytic set
 (real or complex) satisfy condition $(n)$ and are normally pseudo-flat along
each stratum. Twenty years later  Henry et Merle \cite{hm2}
 obtained $ (n)$ with $S$ replaced
by $X \cup Y $ when $X$ and $Y$ are two adjacent strata of a sub-analytic
Whitney stratification of $X \cup Y $. In \cite{ot3}  Orro and the first author  introduced a metric invariant of Kuo's ratio test $(r)$ (introduced by  Kuo \cite{k}  in 1971), denoted by $(r^e)$, for $e \in [0,1]$.
Every $C^2$ $(w)$-regular stratification satisfies automatically $(a)$ and
$(r^e)$, i.e. $(a + r^e)$. Here as usual $(w)$ refers to the Kuo-Verdier condition \cite{ve} which implies trivially Kuo's ratio test $(r)$. Recall that $(r)$ implies Whitney $(b)$-regularity for subanalytic strata \cite{k}. In Proposition 3.2 below we show that $(r)$ implies $(b)$ for definable sets in any o-minimal structure, showing that $(w)$  also implies $(b)$. This is not the case for general stratified sets.
For subanalytic strata it was observed in \cite{ot3} that the combination $(a + r^e)$
is equivalent to Kuo's ratio test $(r)$, and the proof goes through without difficulty for definable stratifications in polynomially bounded o-minimal structures; by \cite{t1} we know that $(r)$
is strictly weaker than $(w)$ in the semi-algebraic case, and there are even
real algebraic examples (\cite{bt}, \cite{Paris7}, \cite{n}).

The equivalence of $(b), (r)$ and $(w)$
for complex analytic stratifications is a consequence of Teissier's programme of characterising $(b)$-regularity
  by equimultiplicity of polar varieties, completed in 1982
(\cite{te1}, \cite{hm1}). One can consult \cite{te2} and \cite{me} for detailed surveys of the complex equisingularity theory.

In \cite{ot2} it was proved that every $(a + r^e)$-regular stratification is
normally pseudo-flat and satisfies condition $(n)$, without subanalyticity. Hence for $(r)$-regular
stratifications which are definable in a polynomially bounded o-minimal
structure, $(n)$ and $(npf)$ hold. We shall show in section 4 that this is not the case for o-minimal structures which are not polynomially bounded.

\medskip
We recall below, for the convenience of the reader, the definitions of the conditions $(a)$ and $(b)$ of Whitney,
$(r)$ of Kuo \cite{k}, $(r^e)$ of Orro-Trotman and $(w)$ of Kuo-Verdier \cite{ve}. We also recall Teissier's notion of $(E^{\star})$ where $(E)$ is any equisingularity criterion.

Let $X$ and $Y$ be two submanifolds of $R^n$ such that $X \subset cl(Y)$ , and let
$\pi$ be a local retraction onto $X$. Following Hironaka \cite{h}, we denote by
$\alpha_{Y,X}(y)$ the distance of $T_yY$ to $T_{\pi(y)}X$, which is

$$\alpha_{Y,X}(y) = max \{< \mu(u), \mu (v) > : u \in N_yY \setminus \{0\}, v \in T_{\pi(y)}X \},$$

\noindent and we denote by $\beta_{Y,X}(y)$ the angle of $y\pi(y)$ to $T_yY$ expressed as

$$\beta_{Y,X}(y) = max \{< \mu(u), \mu(y\pi(y)) \} > : u \in N_yY \setminus \{0\} \},$$

\noindent where $<,>$  is the scalar product on $R^n$.
For $v \in R^n$, the distance of the vector $v$ to a plane $B$ is

$$\eta(v,B) = sup \{< v,n >: n \in B^{\perp}, \vert \vert n \vert \vert = 1\}.$$

Set

$$d(A,B) = sup \{ \eta(v,B) : v \in A, 	\vert \vert v \vert \vert	 = 1 \},$$

\noindent so that in particular $\alpha_{Y,X}(y) = d(T_{\pi(y)}X, T_yY )$.
Set also

$$R_{Y,X}(y) = \frac{\vert \vert y \vert \vert \, \alpha_{Y,X}(y)}{\vert \vert y \pi (y) \vert \vert } \quad {\rm and} \quad W_{Y,X}(y, x) = { \frac{d(T_xX, T_yY )} { \vert \vert yx \vert \vert }}.$$	

\bigskip

\begin{definition}
 The pair of strata $(X, Y )$ satisfies, at $0 \in X$ :

{\it condition $(a)$} if, for $y$ in $Y$,
$$lim_{y \rightarrow 0} \alpha_{Y,X} (y) = 0,$$

{\it condition $(b^{\pi})$} if, for $y $ in $Y$,
$$lim_{y \rightarrow 0} \beta_{Y,X}(y) = 0.$$

{\it condition $(b)$} if, for $y$ in $Y$,
$$lim_{y \rightarrow 0} \alpha_{Y,X}(y) = lim_{y\rightarrow 0} \beta_{Y,X}(y) = 0,$$

{\it condition $(r)$} if, for $y$ in $Y$ ,
$$ lim_{y \rightarrow 0} R_{Y,X}(y) = 0,$$

{\it condition $(w)$} if, for $y$ in $Y$ and $x$ in $X$, $W_{Y,X}(y, x)$ is bounded near
$0$.
\end{definition}
\bigskip

In \cite{ot2} P. Orro and the first author introduced the following condition
of Kuo-Verdier type.

\medskip
\begin{definition} 
Let $e \in [0, 1)$. One says that $(X, Y )$ satisfies condition
$(r^e)$ at $0 \in X $ if, for $y \in Y$, the quantity

$$R^e (y) = \frac{ {\vert\vert \pi (y) \vert \vert }^e \alpha_{Y,X} (y) } {\vert\vert y \pi (y) \vert\vert}.$$

\noindent  is bounded near $0$.
\end{definition}

\medskip
This condition is a $C^2$ diffeomorphism invariant. It is Verdier's
condition $(w)$ when $e = 0$, hence $(w)$ implies $(r^e)$ for all $e \in [0, 1)$. But,
unlike $(w)$, condition $(r^e)$ when $e > 0$ does not imply condition $(a)$ : a
counter-example which is a semi-algebraic surface can be obtained by
pinching a half-plane $\{z \geq 0, y = 0 \}$ of ${\bf R}^3$, with boundary the $x$-axis
$X$ in a cuspidal region $\Gamma = \{x^2 + y^2 < z^p \}$, where $p$ is an odd
integer such that $p > {\frac{2}{e}}$ , such that in $\Gamma$ there are sequences tending to $0$ for which condition $(a)$ fails. Such an example will be $(r^e)$-regular.

\bigskip

\begin{theorem}
{\cite{ot3}}. Every $C^2$ $(a+r^e)$-regular stratification is normally
pseudo-flat and satisfies condition $(n)$.
\end{theorem}

\begin{corollary}
 For $C^2$ $(r)$-regular stratifications which are definable in a
polynomially bounded o-minimal structure, $(n)$ and $(npf)$ hold.
\end{corollary}

\medskip
In section 4 we give a counterexample to the corollary for any non-polynomially bounded o-minimal structure.

\medskip
Now we recall the definition of $E^*$-regularity for $E$ an equisingularity
condition, as in \cite{ot1}. This notion came from the discussion of
B. Teissier in his 1974 Arcata lectures \cite{te}. Teissier stated that one
requirement for an equisingularity condition to be ÒgoodÓ is that it be
preserved after intersection with generic linear spaces containing a given
linear stratum. Various equisingularity conditions have been shown to
have this property, notably Whitney $(b)$-regularity for complex analytic
stratifications (\cite{te1}, \cite{hm1}), and Kuo's ratio test $(r)$, the Kuo-Verdier
condition $(w)$ \cite{nt}, and Mostowski's $(L)$-regularity for subanalytic stratifications \cite{jtv} .

\medskip

\begin{definition} 
Let $M$ be a $C^2$-manifold. Let $X$ be a $C^2$-submanifold
of $M$ and $x \in X$. Let $Y$ be a $C^2$-submanifold of $M$ such that $x \in Y$,
and $X \cap Y = \emptyset$. Let $E$ denote an equisingularity condition (for example $(b), (r)$  or $(w)$). Then $(X, Y )$ is said to be {\it $E_{{\rm cod } \, k}$-regular at $x$}
($0 < k < cod X$) if there is an open dense subset $U_k$ of the Grassmann
manifold of codimension $k$ subspaces of $T_xM$ containing $T_xX$ such that
if $W$ is a $C^2$ submanifold of $M$ with $X \subset W$ near $x$, and $T_xW \in U_k$,
then $W$ is transverse to $Y$ near $x$, and $(X, Y \cap W)$ is $E$-regular at $x$.
\end{definition}

\begin{definition} 
 $(X, Y)$ is said to be {\it $E^*$-regular at $x$} if $(X, Y)$ is $E_{{\rm cod}\, k}$-regular
for all $k, 0 < k < cod X$.
\end{definition}

\begin{theorem}
 {\cite{nt}} For subanalytic stratifications, $(r)$ implies $(r^*)$
and $(w)$ implies $(w^*)$.
\end{theorem}

\medskip
In this sub-analytic case, because $(b)$ implies $(r)$ over $1$-dimensional strata, and $(r)$ always implies $(b)$, we deduce the following corollary.

\medskip

\begin{corollary}
 {\cite{nt}} For subanalytic $(b)$-regular stratifications, $(b^*)$ holds
over every $1$-dimensional stratum.
\end{corollary}

\medskip
In section 4 we give counterexamples to the above Theorem and Corollary in any non-polynomially bounded o-minimal structure.

\bigskip

\section{Normal pseudo-flatness in the non-polynomially bounded case}

In this section we  give several  theorems indicating for which o-minimal structures the definably stratified sets have similar properties to sub-analytic sets, thus avoiding the pathologies exhibited by the functions $f$ and $g$ described in section 4. These theorems are based on earlier work of the second author in  \cite{va1}.

 In the final section of this paper we shall show that condition $(r)$ is too weak  to ensure normal pseudo-flatness in non-polynomially bounded o-minimal structures. In this section we explain how to overcome this problem: we introduce a slightly stronger condition which is enough to entail normal pseudo-flatness. We will also show that the density is continuous along the strata of a stratification satisfying this condition.

We prove in the following proposition (3.1) that one can extend the result of Kuo \cite{k} for semianalytic stratifications that $(b)$ implies $(r)$ along strata of dimension one (extended by Verdier \cite{ve} to subanalytic stratifications), to all stratified sets definable in {\it polynomially bounded} o-minimal structures. The result of Proposition 3.1 is sharp: the example in section 4 below is the first case known of an o-minimal structure with a $(b)$-regular definable stratification which does not satisfy Kuo's ratio test $(r)$. Moreover by Miller's dichotomy \cite{m} every non polynomially bounded o-minimal structure has such an example.

Recall that Kuo's ratio test $(r)$ holds for $(W,X)$ at $(0,0,0)$ if and only if
  $$R(x,y,z) = {\frac{\vert(x,y,z)\vert.
\delta(T_{(x,0,0)}X,T_{(x,y,z)}W)}{\vert(y,z)\vert}} \rightarrow 0$$

\noindent as $(x,y,z)$ tends to $(0,0,0)$ on $W$. For $E$ and $F$ vector subspaces of $\R^n$, the function $\delta$ is defined by:
$$\delta(E,F):=\sup_{|u|=1, u\in E} d(u,F).$$

\begin{proposition}
Let $X$ and $Y$ be $C^2$ disjoint submanifolds which are definable in a polynomially bounded o-minimal structure with $\dim X=1$, and $0 \in X \cap cl(Y)$. If the pair $(Y,X)$ satisfies  the Whitney $(b)$ condition then it also fulfills Kuo's ratio test $(r)$.
\end{proposition}
\begin{proof}
Take such $(Y,X)$ satisfying Whitney's condition $(b)$ at $0 \in X$. After a local change of coordinates we may assume that $X$ is a line. We have to show that the Kuo ratio $R$ tends to zero along any definable arc $\gamma:(0,\ep)\to Y$ with $\gamma(0)=0$. We may assume that $\gamma$ is parameterized by arc-length.   We denote by $P_{t}$ the orthogonal projection onto $T_{\gamma(t)}Y^{\perp}$, and by $\pi$ the orthogonal projection onto $X$.  As $(Y,X)$ is Whitney $(b)$ regular:
$$
\lim_{t\to 0} P_t(\frac{\gamma(t)-\pi((\gamma(t)))}{|\gamma(t)-\pi((\gamma(t))|} )=0.$$

By  L'Hospital's rule and the definability of $\gamma$ (which yields existence of the limits) we have
\begin{equation}\label{eq_gamma_direction}\lim_{t\to 0} \frac{\gamma(t)-\pi(\gamma(t))}{|\gamma(t)-\pi(\gamma(t))|}=\lim_{t\to 0} \frac{\gamma'(t)-\pi(\gamma'(t))}{|\gamma'(t)-\pi(\gamma'(t))|}.\end{equation}

It follows that:
\begin{equation}\label{eq_lim_preuve_w_implique r}\lim_{t\to 0} P_t(\frac{\gamma'(t)-\pi(\gamma'(t))}{|\gamma'(t)-\pi(\gamma'(t))|} )=0. \end{equation}

Because $\gamma(t)$ is a definable map in a polynomially bounded o-minimal structure, there must be  a positive real number $\ep$ such that \begin{equation}\label{eq_gamma}\frac{|\gamma(t)-\pi(\gamma(t))|}{t}\geq \ep |\gamma'(t)-\pi(\gamma'(t))|.\end{equation}

(This is not true for an o-minimal structure which is not polynomially bounded, by Miller's dichotomy \cite{m}. See the analysis of the example below. Or consider $\phi(t) = exp(-C/t)$, which extends continuously at $0$ by $\phi(0) = 0$, and calculate $\phi'(t) =  C \phi(t) /t^2$ so that $\frac{\phi(t)}{t \phi'(t)} = \frac{ t}{C}$ which tends to $0$ as $t$ tends to $0$.)

It is easy to see that if $\gamma$ is not tangent to $X$ at $0$ then by $(a)$-regularity, the ratio test $(r)$ holds. Thus we may assume that $\gamma$ is tangent to $X$. Because also $\gamma$ is parametrised by arc-length, it follows that $ |\gamma'(t) |$ and $ |\pi( \gamma'(t)) |$ tend to 1 as $t$ tends to $0$.

Consequently, because $\gamma(t)$ lies on $Y$, and $
\gamma$ is definable and thus may be assumed to be of class $C^1$ on $(0,\epsilon)$, it follows that $P_t(\gamma'(t)) = 0$, and we have:
$$R(\gamma(t)) = \frac{|P_t(\pi(\gamma'(t)))||\gamma(t)|}{|\gamma(t)-\pi(\gamma(t))| |\pi( \gamma'(t)) | } \approx \frac{|P_t(\gamma'(t)-\pi(\gamma'(t)))|t}{|\gamma(t)-\pi(\gamma(t))| |\pi( \gamma'(t)) |}$$

$$ \approx \frac{|P_t(\gamma'(t)-\pi(\gamma'(t)))|t}{|\gamma(t)-\pi(\gamma(t))|} \leq \frac{1}{\ep}\frac{|P_t(\gamma'(t)-\pi(\gamma'(t)))|}{|\gamma'(t)-\pi(\gamma'(t))|},$$
which tends to zero by (\ref{eq_lim_preuve_w_implique r}).
\end{proof}

\medskip

The converse implication, that $(r)$ implies $(b)$, this time for $Y$ of all dimensions, was also proved by Kuo \cite{k}. Actually Kuo treated the semi-analytic case as his paper predates the introduction of subanalytic sets. Verdier \cite{ve} confirmed that Kuo's proof works also  for subanalytic stratifications. Here we show the new result that the implication is actually valid for definable sets in {\it any} o-minimal structure, not only for those which are polynomially bounded.

\begin{proposition}\label{pro_r_implique_b}
Kuo's ratio test $(r)$ for $(Y,X)$ at $0 \in X$  implies Whitney's condition  $(b)$ at $0$  for any o-minimal structure, and for $X$ of any dimension.
\end{proposition}
\begin{proof}

Fix a pair $(Y,X)$ satisfying $(r)$ at $0 \in X$.
 To prove that $(b)$ holds, take a definable $\gamma: [0,\ep)\to X$ with $\gamma(0)=0$.

 We may assume that $\gamma$ is parameterized by arc-length.

  Then because $P_t (\gamma'(t))\equiv 0$ ($\gamma(t)$ lies on $Y$) we may write using (\ref{eq_gamma}): $$\hskip-15truemm \left\vert P_t\left(\frac{\gamma'(t)-\pi(\gamma'(t))}{|\gamma'(t)-\pi(\gamma'(t))|}\right)\right\vert= \left\vert P_t\left(\frac{\pi(\gamma'(t))}{|\gamma'(t)-\pi(\gamma'(t))|}\right)\right\vert $$

  $$ <    \frac{t|P_t(\pi(\gamma'(t)))|}{|\gamma(t)-\pi(\gamma(t))|}\approx   \frac{|\gamma(t)| |P_t(\pi(\gamma'(t)))|}{|\gamma(t)-\pi(\gamma(t))| | \pi(\gamma'(t)) |} \leq R(\gamma(t)),$$
  using the mean-value theorem, and local monotonicity (due to definability and cell decomposition \cite{dm}, cf. the observation of Proposition 1.10 in \cite{l}) of $\gamma'$ to prove the strict inequality. But $R(\gamma(t))$  tends to zero in virtue of our hypothesis that  $(r)$ holds.  This shows that the angle between $\gamma'-\pi(\gamma')$ and $T_{\gamma(t)} Y$ tends to zero as $t$ goes to zero.   By (\ref{eq_gamma_direction}), the angle between $\gamma-\pi(\gamma)$ and $T_{\gamma(t)} Y$ must tend to zero as well. This establishes Whitney's condition $(b)$ at $0$ for the pair $(Y,X)$.
\end{proof}

The previous proposition has as an immediate corollary that the Kuo-Verdier condition $(w)$ (which trivially implies $(r)$) also implies Whitney's condition $(b)$ for definable stratifications in any o-minimal structure, a result previously proved by Ta L\^e Loi as Proposition 1.10 in  \cite{l}. We remark that $(w)$ does not imply $(b)$ for general $C^2$ stratified sets - consider the topologist's sine curve in the plane, $S = \{ y = sin \frac{1}{x} : x > 0 \} \cup \{ x = 0, - 1 \leq y \leq 1 \}$, taking the two strata $\{ x = 0, - 1 \leq y \leq 1 \}$ and its complement.

\bigskip

  Let $(Y,X)$ be a pair of strata such that $X \subset cl(X)\setminus Y$. As we may work with a coordinate system we shall identify $X$ with a neighborhood of the origin in $\R^k \times 0$. Denote by $\pi:\R^n \to \R^k$ the orthogonal projection.

Define  for $t$ positive small enough:
$$r(t):=\sup_{|\pi(y)|=t} \frac{\delta(X, T_{y} Y)}{|y-\pi(y)|}.$$

It is clear that the Kuo-Verdier condition $(w)$ is equivalent to the boundedness of $r(t)$ for small $t$. Moreover the ratio test $(r)$ is equivalent to the conjunction of $(a)$ and the property that $tr(t)$ tends to zero as $t$ goes to zero. 

\bigskip

\begin{definition}
Let $X$ and $Y$ be two strata with $X \subset cl(X)\setminus Y$. We will say that $(Y,X)$ satisfies {\bf the condition $(r_{\int} )$} at $x \in X$  if $\int_{[0,\ep]} r(t)dt <\infty$.
\end{definition}

 We shall see that, combined with $(a)$-regularity, $(r_{\int})$ implies local topological triviality, normal pseudo-flatness and the continuity of the density, in the case where the strata are definable.

\bigskip

\begin{definition}
    An {\bf  $\alpha$-approximation of the identity} is a  homeomorphism $ h \; : \R^n \times \R^m  \rightarrow \R^n \times \R^m$ of type $h(x;t)=(h_t(x);t)$  with:
 $$|h_t(x)-x| \lesssim \alpha(|x|;t),$$
 $$|h_t^{-1}(x)-x| \lesssim \alpha(|x|;t).$$
\end{definition}

Given a definable set $A$ we set:
$$\psi (A;u) = \hn^k(A \cap B(0;u)),$$
where $k$ denotes the Hausdorff dimension of $A$ and $\hn^k$ the $k$-dimensional Hausdorff measure. Recall that the density (or Lelong number) of $A$ at the origin is then defined as:$$\theta(A):=\lim_{u\to 0} \frac{\psi(A,u)}{\mu_k u^k},$$
where $\mu_k$ stands for the volume of the unit ball in $\R^k$.

\begin{theorem}\label{alpha_approx_invariance_des_vol}(Theorem $4.5$ of \cite{va2})
 Let $A$ and $B$ be two definable (in some o-minimal structure) families of sets   of dimension $l$ of $\R^n \times \R^m$. Let $\alpha:\; \R \times \R^m\rightarrow \R$ be a  function with $\alpha(r,t)\leq\frac{r}{3}$,  for all $t \in \R^m$, and
let $h: A \rightarrow B$ be an $\alpha$-approximation of the identity and let $V$ be a compact subset of $\R^m$.

There is a constant $C$ such that for every $t \in V $ we have for $u>0$ small enough:
$$|\psi(A_t;u)-\psi(B_t;u)| \leq C \alpha(u;t)u^{l-1}.$$
\end{theorem}

The next result is a generalisation of a theorem of Fukui and Paunescu \cite{fp}.

\begin{proposition}\label{p isotopie b}
Let $(A;\Sigma)$ be a definable  stratified set and assume that  $Y:=0\times\R^{k}$ is a stratum. If all   pairs of strata $(X,X') \in \Sigma \times \Sigma$  satisfy conditions $(a)$ and  $\rb$ then  there exist neighbourhoods $U$ and $V$ of the origin in $\R^n$ and $\R^k$ and a topological trivialization
 $$h : A\cap U \to A_0 \cap U\times V,  $$
$(x;t)\mapsto (h_t(x);t)$ such that $h_t:\R^{n-k} \to \R^{n-k}$ is an $\alpha$-approximation of the identity with $\alpha(u;t)= u \int_0 ^{t} r(s)ds$.
\end{proposition}

\begin{proof}
For simplicity, we will assume that $Y$ is one dimensional. We construct an isotopy by integration of vector fields as in
Thom-Mather isotopy lemma \cite{ma}.

Thanks to the  $\rb$  condition  the tangent spaces
satisfy the following estimate:
$$\delta(X; T_y Y)\leq C r(|\pi(y)|)\cdot |y-\pi(y)|$$
 (where $\pi$ is the orthogonal projection onto $X$)  in a sufficiently small
neighborhood of the origin.

By standard arguments \cite{ve,ma,ot2} we may obtain a stratified
unit vector field $v$ defined in a neighbourhood of the origin in  $A$, tangent to $X$ and
 to the strata, smooth on every stratum,  and
 satisfying $\pi(v)=\pa_n$ (where $\pa_n$ generates $X$) as well as:
\begin{equation}\label{champ r^e proof}
|v(y)-v(\pi(y))| \leq C r(|\pi(y)|)\cdot |y-\pi(y)|
\end{equation}
for some constant $C$.

 Denote by $\phi$ the
one-parameter group generated by this vector
 field. Let  $\phi=(\phi_1;\phi_2) \in \R^{n-1}  \times \R $; we also  have  by
  (\ref{champ r^e proof}) (see   \cite{ot3} Lemma $3.4$) a positive constant $C$ such that:
\begin{equation}\label{eq pr isot courbes int}
|\phi_1 (q;-t)-y| \leq C\int_{0} ^{t} r(s) ds \cdot |y|,
\end{equation}
for $q=(y,t) \in Y \subset  \R^{n-1} \times \R $. Existence of integral curves is proved as in \cite{ot2} (see Lemma $3.3$ of the latter article).

The desired trivialization in then given by $h_t(y):= \phi_1(q,-t)$ for $y \in A_t$, $q=(y,t)$.
 By
     (\ref{eq pr isot courbes int}), it is an $\alpha$-approximation of the identity with
     $\alpha(u;t) =u\int_0 ^t r(s)ds$.
\end{proof}

The following theorem is a generalisation of Comte's theorem \cite{c}.

\begin{theorem}\label{thm_density_continuity}
Let $(A;\Sigma)$ be a definable  stratified set and assume that  $X:=0\times\R^{k}  $ is a stratum. If all   pairs of strata $(Y,Y') \in \Sigma \times \Sigma$  satisfy conditions $(a)$ and  $\rb$  then the density of $A$ is continuous along $X.$
\end{theorem}

\begin{proof}
 Set
$$\A :=\{(y;t';t) \in \R^{n-k}  \times \R^k \times \R^k: (y;t+t') \in Y \}.$$
 This defines a family of sets (parameterized by $t \in \R^k$) such that the germ of $\A_t$ at $t$ is the germ of $A$ at $(0_{\R^{n-k}};t)$.  It is thus enough to prove that the function $t \mapsto \theta(\A_t)$ is continuous along the stratum $X.$

 Apply Theorem \ref{p isotopie b} to get an $\alpha$ approximation of the identity $h: A  \to A_0 \times X$ with $\alpha(u;t)=u\int_0 ^t r(s)ds$.
  Observe that the family $H_t(y,t'):=h_{t'}^{-1}(h_{t+t'}(y))$ maps $\A_t$ onto $\A_{0}$. We claim that it is  an $\tilde{\alpha}$-approximation of the identity with $\tilde{\alpha}(u,t)=u\int_{[0,u+|t|]} r(s) ds$.

     Take $(y,t')$ in $B(0,u)$. In particular $|t'|\leq u$. Thus $$|h_{t+t'}(y)-y|\lesssim u \int_{[0,|t'|+|t|]} r(s) ds\leq u \int_{[0,u+|t|]} r(s) ds.$$
     In particular, this shows that $h_{t+t'}(y)$ belongs to $B(0,2u)$ (for $t$ and $t'$ sufficiently small).
Therefore, as $h$ is an $\alpha$-approximation of the identity:$$|h_{t'}^{-1} (h_{t+t'}(y))-h_{t+t'}(y)|\lesssim u \int_{[0,t']} r(s) ds\leq u \int_{[0,u+|t|]} r(s) ds.$$
Together with the preceding estimate this implies the desired inequality. A similar computation holds for $H^{-1}$.

  By theorem \ref{alpha_approx_invariance_des_vol} we have for any $t$ and $u$ small enough:
 $$|\psi(\A_t;u)-\psi(\A_{0};u)| \leq C u^{l}\int_{[0,u+|t|]} r(s) ds,$$
where $l$ is the dimension of $A$ and $C$ is a positive constant.  Dividing by $\mu_k u^l$ and passing to the limit as $u$ tends to zero  we get:
 $$ |\theta(\A_t)-\theta(\A_0)| \leq C \int_{[0,|t|]} r(s) ds,$$
 which tends to zero as $t$ tends to zero.
\end{proof}

\begin{theorem}\label{thm_ppn}
If a stratification, definable in some o-minimal structure (not necessarily polynomially bounded), satisfies both $(a)$ and $\rb$ then it satisfies conditions (npf) and $(n)$.
\end{theorem}
\begin{proof}
Fix a stratum $X$ of a $\rb$-regular stratification. Up to a coordinate system, we may assume that $X$ is $\{0\}\times \R^k$, $k=\dim X$. By Proposition \ref{p isotopie b},  there exists a neighborhood $U$ of the origin (in $A$) and a topological trivialization
 $$h : A \to A_0 \times V,  $$
$(y;t)\mapsto (h_t(y);t)$ such that $h_t:\R^{n-k} \to \R^{n-k}$ is an $\alpha$ approximation of the identity with $\alpha(u;t)=u \int_{[0,|t|]} r(s)ds$. This shows that the variation of the secant line $h_t(y)\pi(h_t(y))$ (in the projective space) tends to zero as  $t$ goes to zero.
\end{proof}

The following proposition can be thought of as a generalisation of Paw\l ucki's theorem \cite{pa}, valid for subanalytic sets and without the $(npf)$ criterion, which is a consequence of $(b)$ for subanalytic strata, by a theorem of Orro and Trotman \cite{ot3}.

\begin{proposition}
Let $X$ and $Y$ be two strata which are definable in an o-minimal structure. Assume that $\dim Y=\dim X+1$ and that $Y\cap B(0,\ep)$ is connected for any $\ep>0$ small enough. Then, $(Y,X)$ satisfies $(b)$  and $(npf)$  if and only if $Y\cup X$ is a $C^1$ manifold with boundary.
\end{proposition}
\begin{proof}
$(b)$ and $(npf)$ are $C^1$ invariants. Hence, the if part is clear. Assume that these conditions hold and let us show that $Y \cup X$ is a $C^1$ manifold with boundary. We can identify $X$ with $\{0\} \times \R^k$.

  We first show that at any point  $x \in X$, $\lim_{y \to x} T_y Y$ exists. As the Whitney  $(b)$ condition implies the Whitney $(a)$ condition, we know that any limit of tangent space contains $T_xX$. By the Whitney $(b)$ condition, it also has to contain the limit of secant lines. Therefore, given a sequence $y_i$ tending to $x$, $\lim T_{y_i} Y$  (if it exists)  is characterized by the limit in the projective space of $y_i \pi(y_i)$  (where $\pi$ is the orthogonal projection onto $X$).

   Set $\gamma(y):=(\frac{y\pi(y)}{|y\pi(y)|},y)$. for $y \in Y$, and $W:=\gamma(Y)$.  As $W$ is a definable set of dimension $k+1$, we have:  $$\dim cl(W)\setminus W \leq k.$$

Moreover,  if we set $$Z:=W\cap (S^{n-1}\times X)$$  then, by the above inequality,
$\dim Z\leq k. $ Therefore, a generic fiber of the map $P:Z \to X$ defined by $(u,x) \mapsto x$ cannot have dimension bigger than $0$.

   By $(npf)$, $P$  is an open map. Hence, $P^{-1}(0)$ must be of dimension $0$ as well.
  As it is connected (since $Y$ is), it must be reduced to one single point.
  Hence, at any point $x$ of $X$ there is a unique limit of tangent space $\tau_x$, as claimed.

   The limiting $l_x$ secant being unique at every point $x$ of $X$, it must vary continuously along $X$. Consequently, $\tau_x$ varies continuously as well. As a matter of fact, for a generic projection, the stratum $Y$ is the graph of a function whose derivative extends continuously. The couple $(cl(Y),X)$ thus constitutes a $C^1$ manifold with boundary near the origin (due to topological triviality it must satisfy the frontier condition).
\end{proof}

It follows that  Theorem \ref{thm_ppn} admits the following corollaries.

\begin{corollary}
Let $X$ and $Y$ be two  (connected)  strata which are definable in an o-minimal structure. Assume that $\dim Y=\dim X+1$ and that $(Y,X)$ satisfies $(a)$ and $\rb$.  Then $X\cup Y$ is a $C^1$ manifold with boundary.
\end{corollary}
\begin{proof}
By Proposition \ref{pro_r_implique_b}, $(Y,X)$ satisfies the Whitney $(b)$ condition. Moreover, by Theorem \ref{thm_ppn}, $X \cup Y$ is normally pseudo-flat.  The result follows from the preceding proposition.
\end{proof}

\begin{corollary}
Let $X$ and $Y$ be two  (connected)  strata which are definable in a polynomially bounded o-minimal structure. Assume that $\dim Y=\dim X+1$ and that $(Y,X)$ satisfies $(b)$.  Then $X\cup Y$ is a $C^1$ manifold with boundary.
\end{corollary}
\begin{proof}
By the proof of \cite{ot2} that $(b)$ implies $(npf)$ in the subanalytic case, one may reduce to the case of $1$-dimensional $X$. But then $(r)$ holds, by Proposition 3.1. By Theorem \ref{thm_ppn}, $X \cup Y$ is normally pseudo-flat, because in the polynomially bounded case $(r)$ implies $\rb$.  The result follows from the preceding proposition.
\end{proof}

Corollary 3.11 is also a generalisation of Paw\l ucki's theorem \cite{pa}.

\section{Paw\l ucki's example.}

In this section we  study  geometric properties of the closure $S_g$ of the graph in ${\bf R}^3$ of the function

 $$g(x,z) = \exp((x^2 + 1) \ln (z)) = z^{x^2 + 1}, \,\, z > 0.$$

\noindent  and compare these with the geometric properties of the closure $S_f$ of the graph of a  function previously studied by the first author and L. Wilson,

 $$f(x,z) = z - \frac{z}{ \ln (z)} \ln (x + \sqrt{x^2 + z^2}), \,\, z > 0.$$

The  natural  stratification with two strata of $S_f$  was shown by the first author and Wilson in \cite{tw} to be a counterexample in o-minimal geometry to several statements known to be true for sub-analytic sets, for example that  Kuo's $(r)$-regular stratified sets \cite{k} are normally pseudo-flat (proved by Orro and the first author in \cite{ot3}), and satisfy $(r^*)$-regularity (proved by  Navarro Aznar and the first author in \cite{nt}).

The set $S_g$ was used in a 1985 paper by Paw\l ucki in \cite{pa} as a counterexample to a possible generalisation of his useful and striking theorem concerning sub-analytic stratified sets with a smooth singular locus of codimension $1$:  such a stratified set is Whitney $(b)$-regular if and only if it is locally a finite union of $C^1$ manifolds with boundary (equal to the singular locus).  For the natural two-strata stratification of $S_g$,  $(b)$-regularity holds and the set is homeomorphic to a closed half-plane, however $S_g$ is {\it not} a $C^1$ manifold with boundary. We shall show here that the graph of $g$ is also a counterexample to the geometric statements proved for sub-analytic sets  in \cite{ot3}, as well as having worse properties than the graph of $f$. For example although the graph of $g$ is Whitney $(b)$-regular over the $1$-dimensional stratum $Ox$, it does not satisfy Kuo's ratio test $(r)$, providing the first such example in $o$-minimal geometry. Kuo \cite{k} proved that no such example exists among semi-analytic stratified sets, and the same proof is valid for subanalytic stratified sets. In Proposition 3.1 below we give a proof for definable sets in any polynomially bounded o-minimal structure.  Moreover we show that $S_g$  can be used to define the first example of a definable stratification in an o-minimal structure which is Whitney $(b)$-regular but whose density (as defined by  Kurdyka and  Raby \cite{kr}, generalising the Lelong number \cite{le}) is not continuous along strata.

 Note that by Miller's dichotomy \cite{m} these examples exist in {\it every} o-minimal structure which is not polynomially bounded.

In $\bf R^{3}$ denote by  $Y$ the graph of the function $f(x,z)$,
for $z > 0$ and $x$ and $z$ small,
and  denote by $W$ the graph  of the function $g(x,z)$,
for $z > 0$ and $x$ and $z$ small.
Denote the $x$-axis by $X$.

\bigskip

\begin{remark} {\it $f(x,z) = - f(-x,z)$, i.e. $f$ is an odd function of $x$, while $g(x,z)$ is obviously an even function of $x$.}
\end{remark}


\medskip

\begin{remark} {\it $X \subset \overline{Y}$, because $\lim_{z
\rightarrow 0} f(x,z) = 0.$}
\end{remark}
{\bf Proof.} Obviously,
$$\lim_{z \rightarrow 0} z = 0, \rm{and} \quad \lim_{z
\rightarrow 0} {\frac{1}{\ln z} } = 0.$$

If $x > c > 0$, then $\vert \ln (x + \sqrt{x\sp{2} + z\sp{2}}) \vert
< \vert \ln
(2c)\vert$, so that
$$\lim_{z \rightarrow 0} z \,\ln (x + \sqrt{x\sp{2} + z\sp{2}}) = 0.$$

By remark 1 we do not need to study the case of $x < 0$.

If both $x$ and $z$ tend to $0$, consider the two cases :

(i)  ${\frac{ z }{ x }} \rightarrow 0.$ Then
$$\vert z \,\ln (x + \sqrt{x\sp{2} + z\sp{2}}) \vert < \vert z \, \ln
(2x) \vert <
\vert x \, \ln (2x) \vert \rightarrow 0 \quad \rm{as} \quad x
\rightarrow 0,$$

(ii) $\frac{ x }{ z }$ is bounded. Then

$$\vert z \,\ln (x + \sqrt{x\sp{2} + z\sp{2}}) \vert < \vert z
\,\ln z \vert
\rightarrow 0 \quad \rm{as} \quad z \rightarrow 0.$$

\medskip
Thus  $X \subset \overline{Y}$.
Clearly  $\lim_{z \rightarrow 0} g(x,z) = 0,$ so that also $X \subset \overline{W}$.
\qed

\medskip
Consider the closed stratified set $S_f$ with two strata $(Y,X)$,
and the closed stratified set $S_g$ with two strata $(W,X)$.

\medskip
In \cite{tw} the following five properties were shown to hold for $S_f$ stratified by $(Y,X)$.

 \medskip

{\bf (1) $(n)$ and $(npf)$ fail for $(Y,X)$ at $(0,0,0)$,

(2) Kuo's ratio test $(r)$ holds for $(Y,X)$ along $X$,

(3) Whitney's condition $(b)$ holds for $(Y,X)$ along $X$,

(4) $(b^*)$ and $(r^*)$ fail for $(Y,X)$ along $X$ at $(0,0,0)$,

(5) the density of $S_f$ is constant, hence continuous, along $X$.}

\medskip

Note that by (1) and (3), a general theorem of \cite{ot1} stating that $(w + \delta)$ implies $(n)$ and $(npf)$, implies in turn that the Kuo-Verdier condition $(w)$ fails for $(Y,X)$ at $(0,0,0)$. Here $(\delta)$ refers to the weak Whitney condition introduced by  Bekka and the first author (see \cite{bt1} and \cite{bt2}), which follows from $(b)$ (as its name suggests).

Note also that (1) and (3) show that $S_f$ provides another (complicated) counterexample to Paw\l ucki's Theorem 0.1 above.

\medskip
We shall prove the following five properties for $S_g$ stratified by $(W,X)$.

{\bf (1) $(n)$ and $(npf)$ fail for $(W,X)$ at $(0,0,0)$,

(2) Kuo's ratio test $(r)$ fails for $(W,X)$ along $X$ at $(0,0,0)$,

(3) Whitney's condition $(b)$ holds for $(W,X)$ along $X$,

(4) $(b^*)$ and $(r^*)$ fail for $(W,X)$ along $X$ at $(0,0,0)$,

(5) while the density of $S_g$ is constant along $X$ at $(0,0,0)$, the density of the $3$-dimensional stratified set defined by the convex hull of $S_g$ and the upper half plane $\{ y = 0, z > 0 \}$ is not continuous along $X$ at $(0,0,0)$.}

\medskip

In particular Property 5 for $S_g$  shows that the theorems of Comte [C] and of the second author \cite{va1,va2},  proved for subanalytic sets, do not hold for general o-minimal structures. This is a surprise:  it gives the first counterexample to continuity of the density along strata of Whitney regular stratified sets definable in some o-minimal structure. This answers negatively a question posed explicitly by the first author and L. Wilson on page 464 of \cite{tw}.

\bigskip

{\bf Property 1.} {\it $(n)$ and $(npf)$ fail for $S_g = W \cup X$ at $(0,0,0)$. }

{\bf Proof. }

We will show that the limits of secants from $(x,0,0)$ to $(x, g(x,z),z)$ as
$(x,z)$ tends to $(x_0,0)$ are the straight lines which in the
$(y,z)$-plane have equations

$ y = 0 \quad {\rm {if}}\quad x_0 > 0$

  $y={\sigma}z \, \,{\rm{for \,\,all}} \, \,
\sigma \in [0,1] \, \, \quad {\rm {if}}\quad x_0 = 0
\quad
(1.1)$

  $y=0 \quad \,{\rm {if}}\quad x_0 < 0.$

However, for the secants from $(0,0,0)$ to $(0,f(0,z),z)$ as $z$ tends to $0$,
the limiting secant is $y = z$.
Hence $(n)$ fails (the tangent cone to $C_0(S_0)$ does not
equal the fibre at $0$ of the normal cone). Moreover $(npf)$ fails
since for $x_0 \neq 0$ the fibre
at $x_0$ of the normal cone is $0$-dimensional, while the fibre at
$0$ is $1$-dimensional.

\medskip

{\it Proof of (1.1).}
First observe that, for all $0<z<1$, the secant from $(0,0,0)$ to
$(0, g(0,z),z)$ has slope
  $${\frac{g(0,z)}{z}} = 1.$$

Take $x_0 > 0$ and
let $(x,z)$ tend to $(x_0,0)$. The slope  of the secant from
$(x,0,0)$ to $(x,g(x,z),z)$ is
$${\frac{g(x,z)}{z}} = z^{x^2} = exp(x^2 \ln z)$$

\noindent which tends to $0$ as
$z$ tends to $0$ and $x$ tends to $x_0$.

By symmetry (Remark 1), when $x_0 < 0$ the limiting slope is also $0$.

Now suppose $(x,z)$ tends to $(0,0)$.

By symmetry (Remark 1 again) it will be enough to study the case $x > 0$ and
to show that all the values $\sigma \in [0,1]$ are realised. So we
must show that the limits of
$z^{x^2}ñ $ take all
  values in $[0,1]$ as $x$ and $z$ tend to $0$ when $x > 0$.

Let $\sigma \in (0,1)$.

 On the curve $z = \exp({\frac{\ln \sigma}{x^2}})$, i.e. $x^2 \ln z = \ln \sigma$,
$\exp(x^2 \ln z) = \sigma$ so in particular the limit as $(x,z)$ tends to $(0,0)$ is $\sigma$.

On the curve $z = \exp({\frac{-1}{x^3}})$,  $\exp(x^2 \ln z) = \exp({\frac{-1}{x}})$ with limit $0$ as $x$ tends to $0$.

This completes the proof of (1.1), and hence the proof of Property 1. \qed

\medskip
Next we shall study Property 2, which is Kuo's ratio test $(r)$. We show that $(r)$ fails for $(W, X)$ at $(0,0,0)$, the condition failing along flat curves of the form
 $z = \exp({\frac{-C}{x^2}})$. That such an example exists is surprising, because we shall see below that Property 3 -  Whitney's condition $(b)$ - does hold, and this is enough to ensure $(r)$ in the subanalytic case along strata of dimension one, as was first shown in 1970 by Tzee-Char Kuo \cite{k}.   Along strata of  higher dimension it is not the case that $(r)$ follows from $(b)$-regularity as was illustrated by the semi-algebraic examples constructed by the first author  during the Nordic Summer School at Oslo in August 1976 \cite{t1}. Real algebraic examples  were given in the first author's 1977 Warwick thesis \cite{t2} and can be found with other real algebraic examples in a joint paper by the first author with Brodersen \cite{bt}. The systematic calculations of $(b)$-regularity by the first author in his 1977 thesis \cite{t2}, completed in  \cite{Paris7} and of $(w)$-regularity in the 1996 thesis of Noirel \cite{n} provide infinitely many real algebraic examples.
 
 \bigskip

{\bf Property 2.} {\it $(r)$ fails for the pair of strata $(W,X)$ at $(0,0,0)$. }

{\bf Proof.}

Recall that Kuo's ratio test $(r)$ holds when
  $$R(x,y,z) = {\frac{\vert(x,y,z)\vert.
\delta(T_{(x,0,0)}X,T_{(x,y,z)}W)}{\vert\vert(y,z)\vert\vert}} \rightarrow 0$$

as $(x,y,z)$ tends to $(0,0,0)$ on $W$.

$$\hskip-38mm{\rm{Now,}} \quad \delta(T_{(x,0,0)}X, T_{(x,y,z)}W) = \quad
 \frac{\vert \frac{\partial g}{\partial x}\vert}{\vert\vert (
\frac{\partial g}{\partial x}, -1,
\frac{\partial g}{\partial z}) \vert\vert}
<  \vert  \frac{\partial g}{\partial x} \vert . $$

Calculating, using $g = \exp((x^2 + 1) \ln z)$,

 $$ {\vert {\frac{\partial g}{\partial x}}\vert} = \vert 2x.\ln z. \exp((x^2 + 1) \ln z) \vert = \vert 2x.z.\ln z. \exp(x^2.\ln z)\vert ,$$

 \noindent  which tends to zero as $z$ tends to $0$, because $z .ln z$ tends to $0$ as $z$ tends to $0$, while $exp(x^2. lnz)$ is bounded above by 1 as $x$ and $z$ tend to $0$.
Further,

 $$ {\vert {\frac{\partial g}{\partial z}}\vert} = \exp(x^2 \ln z),$$

 \noindent  which is bounded between $0$ and $1$ as $x$ and $z$ both tend to zero.

$$\hskip-38mm{\rm{Thus,}} \qquad R(x,y,z) \approx
{\frac{\vert {\frac{\partial g}{\partial x}}\vert .\vert\vert (x,y,z)
\vert\vert}{\vert\vert (y,z) \vert\vert } } \approx {\frac{\vert {\frac{\partial
g}{\partial x}}\vert }{\vert z \vert
}}.\sqrt{x\sp{2} + z\sp{2}} $$

$$  \approx \vert x .z.  \ln z . \exp(x^2  \ln z) .\frac{\sqrt{x\sp{2} + z\sp{2}}}{z}\vert $$

$$  = \vert x .\ln z . \exp(x^2  \ln z) .{{\sqrt{x\sp{2} + z\sp{2}}}}\vert.$$

\medskip

We need to check whether this tends to zero as $z$ tends to $0$.

If $z$ dominates $x$,  then because $z.\ln z$ tends to zero as $z$ tends to zero, we can deduce that $R(x, y ,z)$ also goes to zero.

There remains the case where $x$ dominates $z$ to consider.
In this case $R(x,y,z)$ will be equivalent to $x^2\ln z\exp(x^2 \ln z)$.
Because $x^2 \ln z$ has values running from $0$ to $- \infty$ as $x$ and $z$ tend to zero, we need to study the function $w \exp(w)$ for $w \in ] - \infty, 0]$. This takes the value $0$ when $w = 0$, and also tends to $0$ as $w$ tends to $- \infty$. However we can also choose $x$ and $z$ tending to $0$ so that $w = x^2 \ln z $ tends to a constant $-C$, $C > 0$. In particular one may find $x$ and $z$ tending to zero on the curve $x^2 \ln z = - C$, i.e. $z = \exp(- C/ x^2)$. Such a curve is flat, tangent to the $x$-axis, and the associated curve

$(x, \exp((x^2 + 1) \ln z), z) = (x, \exp(- C (x^2 + 1)/ x^2), \exp( -C/x^2))$

$\hskip37mm= (x, e^{-C} e^{-C/x^2}, e^{-C/x^2})$

\noindent  lies on $W$. The limit of $R$ restricted to such a curve is clearly equivalent to $- C \exp(- C)$, so that the ratio test $(r)$ will fail to hold at $(0,0,0)$ for the pair $(W,X)$. \qed

\bigskip
{\bf Note.} Condition $(a)$ holds.
As above, $d(T_{(x,0,0)}X,
T_{(x,y,z)}W) <  \vert
{\frac{\partial g}{\partial x}}\vert , $
and we saw that $\vert
{\frac{\partial g}{\partial x}}\vert $ tends to $0$ as $z$ tends to $0$, giving $(a)$-regularity.

\bigskip

By the main theorem of \cite{ot2}, the condition
$(r^e)$ (defined in \cite{ot2}) must fail for $(W,X)$ at the origin, because $(a)$ holds, while $(n)$ and $(npf)$ fail.

\bigskip

{\bf Property 3.} {\it $(b)$ holds for $(W,X)$ at $(0,0,0)$.}

\medskip

{\bf Proof}.  We have just seen that $(a)$ holds. Thus we need only
prove that $(b^{\pi})$ holds.

By remark 1, we need only treat the case $x \geq 0$.

Suppose $0 < z < 1$, and $0 \leq x$, for $x$ small.

To prove that Whitney $(b^{\pi})$ holds at $(0,0,0)$ we
must show that

$$Q(x,z) =  {\frac{-y + zg_z}{\vert\vert(y,z)\vert\vert. \vert\vert (g_x, -1, g_z) \vert\vert}}$$

\noindent tends to zero as $(x,z)$ tends to $(0,0)$.

Now $Q(x,z) = {\frac{-y + z^{x^2+1}(x^2+1)}{\vert\vert(y,z)\vert\vert . \vert\vert(g_x,-1,g_z)\vert\vert}}  = {\frac{x^2 z^{x^2+1}}{\vert\vert(y,z)\vert\vert . \vert\vert(g_x,-1,g_z)\vert\vert}}$

\noindent which is equivalent to $x^2 z^{x^2}$, which tends to zero as $x$ and $z$ tend to $0$.

This
implies that $(b^{\pi})$
holds, and hence that $(b)$ holds for $(W,X)$ on a neighborhood of
$(0,0,0)$ in $X$. \qed

\bigskip

{\bf Property 4.} {\it $(b^*)$ and $(r^*)$ fail for $(W,X)$ at $(0,0,0)$. }

{\bf Proof.}
We intersect $W$ with planes $\{y = az, 0 <a<1\}$ to
  obtain $$ az = \exp((x^2 + 1) \ln z) = z z^{x^2},$$
 which becomes
$$a = \exp(x^2 \ln z) .$$

Hence $\ln a = x^2  \ln z$, $x^2 = {\frac{\ln a}{\ln z}}$, and $z = e^{{\frac{\ln a}{x^2}}}$.

This curve in the $xz$-plane passes through $(0,0)$ if $0 < a < 1$,
 since $Y \cap \{y =
az\}$ contains curves passing
through $(0,0,0)$.  It follows that  $(W\cap\{y=az\}, X)$ cannot be
$(b)$-regular,  and  then by definition $(W,X)$ is not $(b_{cod \,1})$-regular and $(b^*)$ fails.
Note that $(a_{cod 1})$ holds for $(W,X)$, as does $(a)$, and thus $(a^*)$ holds.
 However $(r^*)$ fails because $(r)$ fails at $(0,0,0)$. \qed

\bigskip

{\bf Property 5.} {\it While the density of $S_g$ is continuous along $X$ at $(0,0,0)$, the density of the $3$ dimensional stratified set $K_g$, defined by the convex hull of the union of $S_g$ with the upper half plane $y = 0, z > 0$, is not continuous along $X$ at $(0,0,0)$.}

{\bf Proof.}
By the proof of Property 1, together with $(a)$-regularity, we see that  the pure tangent cone at each $(x_0,0,0)$, where $x_0 \neq 0$, is $\{y = 0, z \geq 0 \}$. The density of $S_g$ at such points is thus $1/2$ according to the definition and results of Kurdyka and Raby \cite{kr}, while the density of $K_g$ at such points on $X$ is $0$.

The set of limiting secants to $S_g$ from the origin $(0,0,0)$ (defining the tangent cone to $S_g$ at $(0,0,0)$)  is precisely  $\{ y= z \}$, so that by the formula for the density of Kurdyka and Raby in terms of the pure tangent cone \cite{kr}, we see that the density at the origin of $S_g$ is $1/2$, while the density at the origin  of $K_g$ is  $1/8$ (the area $\pi/8$ of the part of the sector between $y = 0$ and $y = z$ inside the unit ball centred at the origin, divided by the area $\pi$ of the unit disk). It follows that the density of $S_g$ is constant along $X$, and the density of $K_g$ along $X$ jumps at $(0,0,0)$.
 \qed

\bigskip
\begin{remark}The natural stratification of $K_g$ by dimension is Whitney $(b)$-regular, with one connected stratum of dimension $3$, two strata of dimension $2$ - the graph of $g$ and the upper half-plane $\{ y = 0, z >0 \}$ - and a single stratum of dimension $1$, namely the $x$-axis. Using the same technique as above to obtain a $3$-dimensional stratified set $K_f$ equal to the convex hull of the union of $S_f$ and the upper half-plane $\{ y = 0, z >0 \}$, we find that its natural stratification by dimension has two strata of dimension $1$, namely  the $x$-axis and the half-line $\{ y=0,x=0\}$, so forcing the origin to be a stratum. Hence the density of $K_f$ is continuous along strata.
\end{remark}

\bibliographystyle{amsalpha}

\end{document}